\documentclass[11pt]{article}
\usepackage[a4paper,margin=2.6cm]{geometry}
\usepackage[utf8]{inputenc}
\usepackage{amsmath,amssymb,amsthm}
\usepackage{enumitem}
\usepackage{needspace}
\usepackage[hidelinks]{hyperref}
\hypersetup{pdftitle={The upper semi-Fredholm radius and subprojectivity: lifting and quotient obstructions},pdfauthor={Abdelhalim Azzouz}}

\newtheorem{theorem}{Theorem}[section]
\newtheorem{proposition}[theorem]{Proposition}
\newtheorem{lemma}[theorem]{Lemma}
\newtheorem{corollary}[theorem]{Corollary}
\newtheorem*{theoremA}{Theorem A}
\newtheorem*{theoremB}{Theorem B}
\newtheorem*{theoremC}{Theorem C}

\theoremstyle{definition}
\newtheorem{remark}[theorem]{Remark}
\newtheorem{problem}[theorem]{Problem}
\numberwithin{equation}{section}

\newcommand{\cL}{\mathcal L}
\newcommand{\inn}{\operatorname{in}}
\newcommand{\Ran}{\operatorname{Ran}}
\newcommand{\dist}{\operatorname{dist}}
\newcommand{\ext}{\operatorname{ext}}
\newcommand{\LSF}{\Lambda_{\mathrm{SF}}^{\mathrm{all}}}
\newcommand{\LUSP}{\Lambda_{\mathrm{usp}}}
\newcommand{\StrictS}{\mathcal{SS}}
\newcommand{\sub}{\mathcal S}

\title{The upper semi-Fredholm radius and subprojectivity:\\ lifting and quotient obstructions}
\author{Abdelhalim Azzouz\thanks{Department of Mathematics, Faculty of Science, University Salhi Ahmed, Naama, Algeria. E-mail: \texttt{abdelhalim.azzouz.cus@gmail.com}}}
\date{}

\begin{document}
\maketitle

\begin{abstract}
For a bounded operator $T:X\to Y$, let $d_+(T)$ be its distance to the operators that are not upper semi-Fredholm, and let $\inn(T)=\inf_M\|T|_M\|$ over closed infinite-dimensional subspaces. We study the optimal constant $\LSF(X)$ in $d_+\le C\inn$ for a fixed domain and all targets. A quantitative lifting obstruction yields a complete characterization on $\mathcal L_1$-spaces: finite comparison, uniform subprojectivity, subprojectivity, and absence of infinite-dimensional reflexive subspaces are equivalent. The optimal comparison and uniform subprojectivity constants coincide, including the value $\infty$.

On hereditarily indecomposable domains, we characterize the comparison constant by bounded approximation of quotient maps through strictly singular endomorphisms, with an exact dual formula over $\mathbb C$. If strictly singular endomorphisms are strictly cosingular, identities into quotient renormings have $\inn=\varepsilon$ and $d_+=1$. This applies to the original Gowers--Maurey space. Compact products of strictly singular operators, or reflexivity of their ideal modulo compact operators, give further obstructions, covering Argyros--Haydon and Argyros--Motakis spaces. These ratios diverge even for target norms with distortion one. The constant also passes to complemented subspaces without a projection-norm factor. The proofs use an extension formula and reduction to equivalent target norms. The converse for arbitrary domains remains open.
\end{abstract}


\section{Introduction}

A bounded operator with closed range and finite-dimensional kernel remains upper semi-Fredholm under sufficiently small perturbations. How large can such a perturbation be? For $T:X\to Y$, this is measured by
\[
 d_+^{X,Y}(T)=\dist\bigl(T,\cL(X,Y)\setminus\Phi_+(X,Y)\bigr),
\]
where $\Phi_+(X,Y)$ denotes the upper semi-Fredholm operators. This radius, studied by Zem\'anek \cite{Zem84,Zem85}, records a stability property that is usually difficult to compute. Schechter's quantity
\[
 \inn(T)=\inf_{M\in\sub(X)}\|T|_M\|,
\]
where $\sub(X)$ consists of the closed infinite-dimensional subspaces of $X$, offers a more accessible description: it measures how small $T$ can become on an infinite-dimensional subspace \cite{Sch72}. Both quantities vanish precisely outside $\Phi_+$, and $\inn(T)\le d_+(T)$. Their common zero set nevertheless conceals a substantial quantitative difference.

Gonz\'alez and Martin\'on classified the operational quantities associated with semi-Fredholm operators in \cite{GM95}. They subsequently proved that none of the quantities in that classification is equivalent to the corresponding perturbation radius \cite{GM12}. Their upper semi-Fredholm counterexamples use products of hereditarily indecomposable spaces; the construction even gives invertible operators on one fixed complex Banach space for which $\inn/d_+\to0$ \cite[Theorem 3.2 and Proposition 3.4]{GM12}. The same paper gives equality for Hilbert domains and comparison for injective targets \cite[Proposition 3.5]{GM12}. Thus the remaining question concerns the geometry of a fixed domain for which comparison holds uniformly over all targets. Related all-target comparisons of operational quantities characterize the bounded compact approximation property in \cite{CR23}.

We write
\[
 \LSF(X)=\inf\{C\ge1:\ d_+^{X,Y}(T)\le C\inn(T)
 \text{ for every }Y\text{ and every }T\in\cL(X,Y)\}.
\]
A Banach space is subprojective if each of its closed infinite-dimensional subspaces contains a further infinite-dimensional subspace complemented in the ambient space \cite{Whi64}. The uniform version requires a common bound for the projection norms \cite{OiSp15}; we denote its optimal constant by $\LUSP(X)$. Projections immediately give $\LSF(X)\le\LUSP(X)$ (Proposition~\ref{prop:subproj}). We ask whether they are also necessary.

\begin{problem}\label{prob:converse}
If $\LSF(X)<\infty$, must $X$ be uniformly subprojective?
\end{problem}

To see the difficulty, consider a restriction $T|_M$ of small norm. An extension $E:X\to Y$ with $E|_M=T|_M$ makes $T-E$ vanish on $M$ and therefore bounds the radius. Conversely, approximation of a non-upper semi-Fredholm operator by operators with infinite-dimensional kernel gives the exact identity
\[
 d_+(T)=\inf_{M\in\sub(X)}\inf\{\|E\|:\ E\in\cL(X,Y),\ E|_M=T|_M\}.
\]
This formulates the extension mechanism already used in \cite[Proposition 3.5]{GM12}, together with the classical kernel approximation results \cite{GO86,AAG92}. The gap between $\inn$ and $d_+$ is the gap between making a restriction small and extending it with small norm. A projection provides such an extension, but an arbitrary extension need not yield a projection.

Our main argument recovers a projection from a small extension when a suitable quotient admits liftings. For $H\in\sub(X)$ set
\[
 a_X(H)=\inf_{N\in\sub(H)}\lambda_X(N),
\]
where $\lambda_X(N)$ is the infimum of the norms of projections of $X$ onto $N$, with value $\infty$ when $N$ is not complemented. The pair $(X,H)$ has the lifting property if every operator $X\to X/H$ lifts through the quotient map $q_H:X\to X/H$.

\begin{theoremA}
Let $H\in\sub(X)$ be uniformly subprojective for its inherited norm with constant $K$, and suppose that $(X,H)$ has the lifting property. Then
\[
 \LSF(X)\ge a_X(H)/K.
\]
If no infinite-dimensional subspace of $H$ is complemented in $X$, then $\LSF(X)=\infty$.
\end{theoremA}

Theorem~\ref{thm:lifting} gives this result quantitatively. It uses the identities from $X$ into the renormings
\[
 |x|_\varepsilon=\max\{\varepsilon\|x\|,\|q_Hx\|\},
 \qquad 0<\varepsilon\le1.
\]
Their restriction quantity is exactly $\varepsilon$. When $a_X(H)=\infty$, their radii have an explicit lower bound of order $\sqrt\varepsilon$. Complemented witnesses are included, but the lifting property also permits uncomplemented witnesses. This is what makes the result useful on $\mathcal L_1$-spaces: the lifting principle of Lindenstrauss \cite{Lin64,KP97} applies to reflexive kernels, whereas the Dunford--Pettis property prevents those kernels from containing infinite-dimensional complemented subspaces.

\Needspace{12\baselineskip}
\begin{theoremB}
For an infinite-dimensional $\mathcal L_1$-space $X$, over either scalar field, the following are equivalent:
\begin{enumerate}[label=\textup{(\roman*)},nosep]
\item $\LSF(X)<\infty$;
\item $X$ is uniformly subprojective;
\item $X$ is subprojective;
\item $X$ contains no infinite-dimensional reflexive subspace.
\end{enumerate}
Moreover $\LSF(X)=\LUSP(X)$, with the value $\infty$ allowed. Their common finite value is at most the distortion of any embedding of $X$ into an $L_1(\mu)$.
\end{theoremB}

This is Theorem~\ref{thm:L1}. Its proof combines the lifting obstruction with classical results of Aldous, Rosenthal, James, and Dor. On $L_1(0,1)$ a Gaussian subspace supplies the witness, and Corollary~\ref{cor:L1explicit} gives
\[
 \inn(J_\varepsilon)=\varepsilon,\qquad
 d_+(J_\varepsilon)\ge\frac{2}{1+\sqrt{5+4/\varepsilon}}.
\]
In contrast, $\LSF(\ell_1(\Gamma))=1$. Thus the comparison detects a geometric difference between atomic and non-atomic $L_1$ spaces despite their common local $\mathcal L_1$ structure. Theorem~\ref{thm:conditional} gives a converse beyond this class when every subspace contains an internally uniform witness with the lifting property. The real sum $(\bigoplus_n L_{p_n}(0,1))_{\ell_2}$, with $p_n\to\infty$, shows that ordinary subprojectivity alone does not suffice.

The comparison on H.I. domains requires a different argument, since their subspaces provide no internally uniformly subprojective witnesses. On a complex H.I. space, every endomorphism in the initial norm satisfies $d_+=\inn$ (Corollary~\ref{cor:HI-endomorphisms}). The all-target question nevertheless allows the target to carry an independent equivalent norm. For these domains the complement of $\Phi_+$ is the strictly singular ideal, so the comparison can be studied through the distance to that ideal.

\Needspace{13\baselineskip}
\begin{theoremC}
Let $X$ be H.I. Then $\LSF(X)=\infty$ under any of the following hypotheses:
\begin{enumerate}[label=\textup{(\roman*)},nosep]
\item every strictly singular endomorphism is strictly cosingular;
\item every product of $m$ strictly singular endomorphisms is compact, for some $m\ge1$;
\item the Banach space $\StrictS(X)/\mathcal K(X)$ is reflexive.
\end{enumerate}
Under \textup{(i)}, for every closed $M$ of infinite dimension and codimension, the identity into the norm $\max\{\varepsilon\|x\|,\dist(x,M)\}$ satisfies $\inn=\varepsilon$ and $d_+=1$ for $0<\varepsilon\le1$.
\end{theoremC}

Theorems~\ref{thm:HI-cosingular}, \ref{thm:HI-nilpotent} and \ref{thm:HI-reflexive-ideal} establish these obstructions. Ferenczi's quotient indecomposability result verifies the first hypothesis for the original Gowers--Maurey space \cite{Fer99}. The scalar-plus-compact property gives the Argyros--Haydon example \cite{AH11,KL17}, while compact products of strictly singular operators give the Argyros--Motakis examples \cite{AM14,AM16}. Thus the radius on a single H.I. domain can be unbounded relative to $\inn$ even though the comparison for its initial endomorphism norm is exact.

For an arbitrary H.I. space, Theorem~\ref{thm:HI-quotients} expresses $\LSF(X)$ as $\sup_M b_X(M)$, where $b_X(M)$ measures bounded approximation of $q_M$ by $q_MS_n$ with $S_n$ strictly singular. Theorem~\ref{thm:HI-dual} identifies $b_X(M)$ with the reciprocal of a distance in $\cL(X)^*$ over $\mathbb C$. These formulas turn the remaining H.I. question into a precise approximation problem. The quotient renormings used in Theorem C have distortion one, so their divergent radii arise from a global extension obstruction even when the norms agree up to a scalar on some infinite-dimensional subspace.

Theorem~\ref{thm:renorm} reduces the all-target comparison to identities into equivalent renormings of the fixed domain. Proposition~\ref{prop:complemented-inheritance} passes the same constant to complemented subspaces. Section~\ref{sec:converse} proves the $\mathcal L_1$ characterization and the conditional converse; Section~\ref{sec:HI} develops the quotient formulas and H.I. obstructions. Section~\ref{sec:questions} presents the remaining questions for H.I. spaces, classical domains and the lower semi-Fredholm radius.

\subsection{Notation and elementary facts}\label{sec:prelim}
Spaces are over a fixed field $\mathbb K=\mathbb R$ or $\mathbb C$, and the domain $X$ is infinite-dimensional. For $R\in\cL(X,Y)$ let $j(R)=\inf_{\|x\|=1}\|Rx\|$. For a closed subspace $N\subset X$,
\[
\lambda_X(N)=\inf\{\|P\|:\ P\in\cL(X),\ P^2=P,\ \Ran P=N\},
\]
with $\inf\varnothing=\infty$. The uniform subprojectivity constant is
\begin{equation}\label{eq:usp-def}
 \LUSP(X)=\sup_{M\in\sub(X)}a_X(M)=\sup_{M\in\sub(X)}\inf_{N\in\sub(M)}\lambda_X(N).
\end{equation}
A constant $K$ is \emph{admissible} if every $M\in\sub(X)$ contains an $N\in\sub(M)$ and a projection $X\to N$ of norm at most $K$. Every $K>\LUSP(X)$ is admissible, but attainment at the infimum is not assumed. A closed subspace $H$ is \emph{uniformly subprojective with constant $K$} if $K$ is admissible for $H$ with its inherited norm. This property passes to closed infinite-dimensional subspaces of $H$ with the same constant, by restricting projections, and it is preserved by isomorphisms up to the factor $\|U\|\|U^{-1}\|$. Direct sums $E\oplus_\infty F$ carry the norm $\max\{\|x\|,\|y\|\}$.

An operator belongs to $\Phi_+$ if and only if it is bounded below on a closed subspace of finite codimension. We use Mazur's device in the following form: \emph{if $F\subset X$ is finite-dimensional and $0<\delta<1$, there is a closed subspace $W$ of finite codimension with $\|y\|\le(1-\delta)^{-1}\|y+w\|$ for $y\in F$, $w\in W$.} Indeed, if $z_1,\dots,z_k$ is a $\delta$-net of the unit sphere of $F$, $f_i$ are norm-one functionals with $f_i(z_i)=1$ and $W=\bigcap_i\ker f_i$, then for a unit vector $y\in F$ with $\|y-z_i\|\le\delta$ we get $\|y+w\|\ge|f_i(y+w)|=|f_i(y)|\ge1-\delta$.

\section{The extension formula and the renorming reduction}\label{sec:extension}

\begin{lemma}\label{lem:kernel}
For $T\in\cL(X,Y)$ put $\kappa_\infty(T)=\inf\{\|E\|:\dim\ker(T-E)=\infty\}$. Then $d_+(T)=\kappa_\infty(T)$. More precisely, if $S\notin\Phi_+$ and $\delta>0$, there is a nuclear operator $V$ with $\|V\|<\delta$ and $\dim\ker(S-V)=\infty$.
\end{lemma}
\begin{proof}
Operators with infinite-dimensional kernel are not in $\Phi_+$, so $d_+\le\kappa_\infty$. Let $S\notin\Phi_+$ with $\dim\ker S<\infty$ (otherwise take $V=0$). Then $S$ is not bounded below on any closed subspace of finite codimension. Put $W_1=X$ and choose a unit vector $x_1$ with $\|Sx_1\|<\delta2^{-3}$. For $n\ge2$, put $W_n=W\cap W_{n-1}$, where $W$ is given by Mazur's device for $\operatorname{span}\{x_1,\dots,x_{n-1}\}$ with parameter $1/2$, and choose a unit vector $x_n\in W_n$ with $\|Sx_n\|<\delta2^{-n-2}$. Since $x_{m+1},\dots,x_n\in W_{m+1}$, we have $\|\sum_{k\le m}a_kx_k\|\le2\|\sum_{k\le n}a_kx_k\|$ for $m<n$. Thus $(x_n)$ is basic with constant at most $2$, its biorthogonal functionals have norm at most $4$, and by the Hahn--Banach theorem they extend to $f_n\in X^*$ with $\|f_n\|\le4$ and $f_n(x_m)=\delta_{nm}$. The operator $V=\sum_n f_n\otimes Sx_n$ is nuclear, $\|V\|<\delta$, and $(S-V)x_m=0$ for all $m$. For general $T$ and $\eta>d_+(T)$, apply this to $S=T-E$ with $\|E\|<\eta$ and $T-E\notin\Phi_+$; then $\kappa_\infty(T)\le\|E+V\|<\eta+\delta$.
\end{proof}

Lemma~\ref{lem:kernel} is a quantitative form of the instability of non-semi-Fredholm operators under compact perturbations \cite{GO86,AAG92}. For $M\in\sub(X)$ and $R\in\cL(M,Y)$ put $\ext_X(R)=\inf\{\|E\|:\ E\in\cL(X,Y),\ E|_M=R\}$.

\begin{proposition}\label{prop:extension}
For every $T\in\cL(X,Y)$,
\[
d_+(T)=\inf_{M\in \sub(X)}\ext_X(T|_M).
\]
\end{proposition}
\begin{proof}
If $E$ extends $T|_M$, then $M\subset\ker(T-E)$; conversely, if $\dim\ker(T-E)=\infty$, then $E$ extends $T|_M$ for $M=\ker(T-E)$. So the right-hand side is $\kappa_\infty(T)$.
\end{proof}

Thus $\inn(T)$ is the infimum of the norms of restrictions of $T$, and $d_+(T)$ the infimum of the costs of extending them back to $X$. For a closed $N\subset X$, an extension of $I_N$ to an operator $X\to N$ is a projection onto $N$; this is how projection constants enter.

\begin{proposition}\label{prop:subproj}
If $X$ is uniformly subprojective with constant $\lambda$, or if $Y$ is $\lambda$-injective, then $\inn(T)\le d_+(T)\le\lambda\inn(T)$ for $T\in\cL(X,Y)$. Hence $\LSF(X)\le\LUSP(X)$. Moreover $\LUSP(X)=1$, and so $d_+=\inn$, when $X$ is a Hilbert space, $\ell_p(\Gamma)$ with $1\le p<\infty$, or $c_0(\Gamma)$.
\end{proposition}
The Hilbert-domain and injective-target statements are \cite[Proposition 3.5]{GM12}; the argument below also gives the uniform subprojectivity formulation. A $\lambda$-injective target means that every bounded operator into it, defined on a closed subspace of any Banach space, extends with norm at most $\lambda$ times its original norm.
\begin{proof}
If $E|_M=T|_M$ then $\|T|_M\|\le\|E\|$, which gives the left inequality. For the right one fix $M\in\sub(X)$. In the first case choose $N\in\sub(M)$ and a projection $P$ of $X$ onto $N$ with $\|P\|\le\lambda$; then $TP$ extends $T|_N$ and $\|TP\|\le\lambda\|T|_M\|$. In the second case extend $T|_M$ with norm at most $\lambda\|T|_M\|$. Proposition~\ref{prop:extension} gives $d_+(T)\le\lambda\|T|_M\|$.

Hilbert spaces have constant $1$. Let $X=\ell_p$ ($1\le p<\infty$) or $c_0$ with unit vector basis $(e_k)$, $M\in\sub(X)$, $0<\theta<1$. Put $n_0=0$ and choose inductively unit vectors $y_k\in M$ vanishing on the first $n_{k-1}$ coordinates and integers $n_k>n_{k-1}$ such that the tail of $y_k$ beyond $n_k$ has norm at most $\theta2^{-k-3}$. Let $u_k$ be the normalized restriction of $y_k$ to $(n_{k-1},n_k]$; then $\|y_k-u_k\|\le\theta2^{-k-2}$. With $u_k^*$ the norm-one functional supported on the same block and $u_k^*(u_k)=1$, $Px=\sum_ku_k^*(x)u_k$ is a projection of norm one onto $[u_k]$. The operator $U=I+\sum_ku_k^*\otimes(y_k-u_k)$ satisfies $\|U-I\|\le\theta$ and $Uu_k=y_k$, so $UPU^{-1}$ is a projection onto $[y_k]\subset M$ of norm at most $(1+\theta)/(1-\theta)$. For $\ell_p(\Gamma)$ and $c_0(\Gamma)$, a separable $M_0\in\sub(M)$ is supported by a countable $\Gamma_0$, and $\ell_p(\Gamma_0)$, $c_0(\Gamma_0)$ are $1$-complemented by restriction.
\end{proof}

\begin{lemma}\label{lem:finite}
If $R:X\to Y$ has finite rank, then $\inn(T+R)=\inn(T)$ and $d_+(T+R)=d_+(T)$. If $F$ is finite-dimensional and $i:Y\to Y\oplus_\infty F$ is the inclusion, then $\inn(iT)=\inn(T)$ and $d_+^{X,Y\oplus_\infty F}(iT)=d_+^{X,Y}(T)$.
\end{lemma}
\begin{proof}
Applying the finite-codimensional criterion on subspaces of $\ker R$ shows $T+R\in\Phi_+\iff T\in\Phi_+$, so translation by $R$ preserves the complement of $\Phi_+$. For $M\in\sub(X)$, $\inn(T)\le\|T|_{M\cap\ker R}\|\le\|(T+R)|_M\|$; take the infimum and exchange $T$ and $T+R$. A perturbation $B=(B_1,B_2):X\to Y\oplus_\infty F$ has $B_2$ of finite rank and $\|B_1\|\le\|B\|$, so $iT-B\in\Phi_+\iff T-B_1\in\Phi_+$; perturbations $iD$ give the reverse inequality. The statement on $\inn$ follows from the isometric inclusion.
\end{proof}

For an equivalent norm $|\cdot|$ on $X$ let $X_{|\cdot|}$ be the renormed space and $J_{|\cdot|}:X\to X_{|\cdot|}$ the identity; the domain keeps its norm.

\begin{theorem}\label{thm:renorm}
For $C\ge1$, the inequality $d_+^{X,Y}(T)\le C\inn(T)$ holds for all $Y$ and $T\in\cL(X,Y)$ if and only if it holds for all identities $J_{|\cdot|}$ into equivalent renormings of $X$. Thus
\begin{equation}\label{eq:renorm}
 \LSF(X)=\sup_{|\cdot|\sim\|\cdot\|}\frac{d_+(J_{|\cdot|})}{\inn(J_{|\cdot|})}.
\end{equation}
\end{theorem}
\begin{proof}
Let $T\in\Phi_+(X,Y)$ (otherwise $d_+(T)=0$), let $P_0$ be a projection onto $\ker T$, and define $T'x=(Tx,P_0x)\in E=Y\oplus_\infty\ker T$. If $\|w\|\le c\|Tw\|$ on $\ker P_0$, then $\|x\|\le(c+1)\|T'x\|$, so $Z=T'(X)$ is closed and $S:X\to Z$, $Sx=T'x$, is an isomorphism onto. By Lemma~\ref{lem:finite}, $\inn(S)=\inn(T)$ and $d_+^{X,E}(T')=d_+^{X,Y}(T)$. An operator with values in $Z$ is upper semi-Fredholm into $Z$ if and only if it is so into $E$, hence $d_+^{X,Y}(T)=d_+^{X,E}(T')\le d_+^{X,Z}(S)$. Finally $|x|=\|Sx\|$ is an equivalent norm, $S=UJ_{|\cdot|}$ with $U$ a surjective isometry, and composition with $U$ preserves $d_+$ and $\inn$. The identity of $(X,\|\cdot\|)$ has $\inn=d_+=1$, so the supremum is at least $1$.
\end{proof}

\begin{proposition}\label{prop:complemented-inheritance}
If $H$ is a closed infinite-dimensional complemented subspace of $X$, with its inherited norm, then
\[
 \LSF(H)\le\LSF(X).
\]
\end{proposition}
\begin{proof}
Choose a bounded projection $P$ onto $H$, write $Z=\ker P$, and for $T:H\to Y$ and $t>0$ define
\[
 T_tx=(TPx,t(I-P)x)\in Y\oplus_\infty Z.
\]
Then $\inn(T_t)\le\inn(T)$. For a perturbation $E:X\to Y\oplus_\infty Z$ of norm $r<t$, all four blocks obtained by restricting to $H$ or $Z$ and projecting on target coordinates have norm at most $r$. The block $D=tI_Z-E_{22}$ is invertible and $\|D^{-1}\|\le(t-r)^{-1}$. Triangular elimination shows that $T_t-E$ is upper semi-Fredholm if and only if its Schur complement
\[
 T-E_{11}-E_{12}D^{-1}E_{21}
\]
is upper semi-Fredholm. If it is not, then $d_+^{H,Y}(T)\le r+r^2/(t-r)$. Perturbations with $r\ge t$ also satisfy the resulting lower bound, so
\[
 d_+^{X,Y\oplus_\infty Z}(T_t)
 \ge\frac{t\,d_+^{H,Y}(T)}{t+d_+^{H,Y}(T)}.
\]
If $\LSF(X)=C<\infty$, combine this with $d_+(T_t)\le C\inn(T_t)\le C\inn(T)$ and let $t\to\infty$. The case $\LSF(X)=\infty$ is immediate.
\end{proof}

No factor involving $\|P\|$ appears in this estimate. Finiteness is also invariant under isomorphisms: if $U:X\to W$ has distortion $D=\|U\|\|U^{-1}\|$, composing both operators and perturbations gives
\[
 D^{-1}\LSF(W)\le\LSF(X)\le D\LSF(W).
\]

\section{The lifting obstruction}\label{sec:lifting}

For a closed subspace $H\subset X$ let $q_H:X\to X/H$ be the quotient map. If $(X,H)$ has the lifting property, the map $\cL(X)\to\cL(X,X/H)$, $L\mapsto q_HL$, is a bounded linear surjection, so by the open mapping theorem there is a constant $C$ such that every $T\in\cL(X,X/H)$ has a lifting $L$ with $q_HL=T$ and $\|L\|\le C\|T\|$. Any such $C$ is called a lifting constant of $(X,H)$. We may and shall choose $C\ge1$; it is fixed independently of the renorming parameter below.

For $0<\varepsilon\le1$ define the equivalent norm
\begin{equation}\label{eq:norm-eps}
 |x|_\varepsilon=\max\{\varepsilon\|x\|,\ \|q_Hx\|\}\qquad(\varepsilon\|x\|\le|x|_\varepsilon\le\|x\|),
\end{equation}
and let $J_\varepsilon:X\to X_\varepsilon=(X,|\cdot|_\varepsilon)$ be the identity. Since $|h|_\varepsilon=\varepsilon\|h\|$ for $h\in H$, the lower inequality in \eqref{eq:norm-eps} and restriction to $H$ give
\begin{equation}\label{eq:in-eps}
 \inn(J_\varepsilon)=\varepsilon.
\end{equation}

\begin{lemma}\label{lem:key}
Let $H\in\sub(X)$ be uniformly subprojective with constant $K$, and let $C$ be a lifting constant of $(X,H)$. Let $0<\varepsilon\le1$ and $0<r<1/C$, and put
\[
 \theta=\frac{KCr\,(r/\varepsilon+Cr)}{1-Cr}.
\]
If $d_+(J_\varepsilon)<r$ and $\theta<1$, then some $N\in\sub(H)$ satisfies
\[
 \lambda_X(N)\le\frac{K\,(r/\varepsilon+Cr)}{1-\theta}.
\]
\end{lemma}
\begin{proof}
By Proposition~\ref{prop:extension} there are $M\in\sub(X)$ and $E\in\cL(X,X_\varepsilon)$ with $E|_M=J_\varepsilon|_M$ and $|Ex|_\varepsilon\le r\|x\|$ for all $x$. Regard $E$ as an operator on $X$. Then $Ex=x$ for $x\in M$, and by \eqref{eq:norm-eps}
\[
 \|E\|\le r/\varepsilon,\qquad \|q_HE\|\le r.
\]
Let $L\in\cL(X)$ be a lifting of $q_HE$ with $\|L\|\le Cr$, and put $A=E-L$. Then $q_HA=0$, so $A$ maps $X$ into $H$, and $\|A\|\le r/\varepsilon+Cr$. For $x\in M$ we have $Ax=x-Lx$ and $\|Ax\|\ge(1-Cr)\|x\|$; hence $A(M)$ is a closed infinite-dimensional subspace of $H$. Choose $N\in\sub(A(M))$ and a projection $Q$ of $H$ onto $N$ with $\|Q\|\le K$.

Let $y\in N$ and write $y=Ax$ with $x\in M$, so that $\|x\|\le\|y\|/(1-Cr)$. Since $Ax=x-Lx$,
\[
 Ay=A(x-Lx)=y-ALx,\qquad QAy=y-QALx,
\]
and $\|QALx\|\le K\|A\|\,Cr\,\|x\|\le\theta\|y\|$. Thus $B=QA|_N\in\cL(N)$ satisfies $\|B-I_N\|\le\theta<1$, so $B$ is invertible with $\|B^{-1}\|\le(1-\theta)^{-1}$. The operator $R=B^{-1}QA:X\to N$ satisfies $Ry=y$ for $y\in N$; it is a projection of $X$ onto $N$ with $\|R\|\le K(r/\varepsilon+Cr)/(1-\theta)$.
\end{proof}

\begin{theorem}\label{thm:lifting}
Let $H\in\sub(X)$ be uniformly subprojective with an admissible constant $K$, and suppose that $(X,H)$ has the lifting property, with lifting constant $C$. Then:
\begin{enumerate}[label=\textup{(\alph*)}]
\item $\displaystyle\liminf_{\varepsilon\downarrow0}\frac{d_+(J_\varepsilon)}{\varepsilon}\ge\frac{a_X(H)}{K}$, and consequently $\LSF(X)\ge a_X(H)/K$;
\item if $a_X(H)=\infty$, then, for $0<\varepsilon\le1$,
\begin{equation}\label{eq:root}
 d_+(J_\varepsilon)\ge r_\varepsilon
 :=\frac{2}{C+\sqrt{C^2+4KC(\varepsilon^{-1}+C)}}.
\end{equation}
In particular $\inn(J_\varepsilon)=\varepsilon$ and
\begin{equation}\label{eq:sqrt}
 \liminf_{\varepsilon\downarrow0}\frac{d_+(J_\varepsilon)}{\sqrt\varepsilon}
 \ge\frac1{\sqrt{KC}},\qquad
 \frac{d_+(J_\varepsilon)}{\inn(J_\varepsilon)}\longrightarrow\infty.
\end{equation}
\end{enumerate}
\end{theorem}
\begin{proof}
For (b), the conditions $Cr<1$ and $\theta<1$ in Lemma~\ref{lem:key} are both satisfied when
\[
 KC(\varepsilon^{-1}+C)r^2+Cr<1.
\]
The positive root of the corresponding equation is $r_\varepsilon$. If $0<r<r_\varepsilon$ and $d_+(J_\varepsilon)<r$, the lemma produces a complemented infinite-dimensional subspace of $H$, contrary to $a_X(H)=\infty$. Letting $r\uparrow r_\varepsilon$ proves \eqref{eq:root}, and its asymptotics give \eqref{eq:sqrt}.

For (a), we may by (b) assume $a=a_X(H)<\infty$. Fix $0<c<a/K$ and put $r=c\varepsilon$. As $\varepsilon\downarrow0$,
\[
 Cr\to0,\quad
 \theta=\frac{KCc\varepsilon(c+Cc\varepsilon)}{1-Cc\varepsilon}\to0,\quad
 \frac{K(c+Cc\varepsilon)}{1-\theta}\to Kc<a.
\]
For small $\varepsilon$, the inequality $d_+(J_\varepsilon)<c\varepsilon$ would give an $N\in\sub(H)$ with $\lambda_X(N)<a$, which is impossible. Thus $d_+(J_\varepsilon)\ge c\varepsilon$ for small $\varepsilon$, and $c<a/K$ was arbitrary. Use \eqref{eq:in-eps} and Theorem~\ref{thm:renorm} to obtain the bound for $\LSF(X)$.
\end{proof}

The finite obstruction also has a bound at each value of $\varepsilon$, rather than only an asymptotic bound.

\begin{proposition}\label{prop:finite-obstruction}
Under the hypotheses of Theorem~\ref{thm:lifting}, suppose $a=a_X(H)<\infty$, and put $b_\varepsilon=\varepsilon^{-1}+C$. Then
\begin{equation}\label{eq:finite-root}
 d_+(J_\varepsilon)\ge
 \rho_\varepsilon(a):=
 \frac{2a}{aC+Kb_\varepsilon+
 \sqrt{(aC+Kb_\varepsilon)^2+4aKC b_\varepsilon(a-1)}}.
\end{equation}
Moreover $\rho_\varepsilon(a)/\varepsilon\to a/K$ as $\varepsilon\downarrow0$, and $\rho_\varepsilon(a)\to r_\varepsilon$ as $a\to\infty$.
\end{proposition}
\begin{proof}
We have $a\ge1$. For $r<r_\varepsilon$, Lemma~\ref{lem:key}, if applicable, gives
\[
 a\le \frac{K b_\varepsilon r(1-Cr)}
 {1-Cr-KC b_\varepsilon r^2}.
\]
Since the denominator is positive, this implies
\[
 a\le(aC+Kb_\varepsilon)r+KC b_\varepsilon(a-1)r^2.
\]
The positive solution of equality is $\rho_\varepsilon(a)$ (also when $a=1$). It is strictly smaller than $r_\varepsilon$, as is seen by substituting $r_\varepsilon$ into the last polynomial. Therefore $d_+(J_\varepsilon)<\rho_\varepsilon(a)$ would allow an $r$ strictly between these numbers and contradict that inequality. The limits follow from the displayed formula.
\end{proof}

If $H$ is complemented by a projection $P$, the formula $s(q_Hx)=(I-P)x$ defines a bounded right inverse of $q_H$ with $\|s\|\le\|I-P\|$, and $L=sT$ lifts any $T\in\cL(X,X/H)$. Theorem~\ref{thm:lifting} therefore contains the obstruction by complemented witnesses; note that the asymptotic bound does not depend on $\|P\|$.

\begin{corollary}\label{cor:complemented}
If $H\in\sub(X)$ is complemented in $X$ and uniformly subprojective for its inherited norm, then $a_X(H)\le\LUSP(H)\,\LSF(X)$. In particular, $\LSF(X)\ge a_X(H)$ whenever $H$ is a complemented subspace isometric to a Hilbert space, and this bound is attained on Hilbert spaces.
\end{corollary}
\begin{proof}
Apply Theorem~\ref{thm:lifting}(a) with every admissible $K>\LUSP(H)$. On a Hilbert space both sides equal $1$ by Proposition~\ref{prop:subproj}.
\end{proof}

\section{Consequences: \texorpdfstring{$\mathcal L_1$}{L1}-spaces and a conditional converse}\label{sec:converse}

\subsection{\texorpdfstring{$\mathcal L_1$}{L1}-spaces}

A Banach space $X$ is an $\mathcal L_1$-space if there is $\sigma<\infty$ such that every finite-dimensional subspace is contained in a finite-dimensional subspace $F$ with $d(F,\ell_1^{\dim F})<\sigma$. This is a local condition; the embedding into a function space needed below is recorded explicitly.

\begin{lemma}\label{lem:embedding}
Every $\mathcal L_1$-space admits an isomorphic embedding $u:X\to L_1(\mu)$ of finite distortion $D=\|u\|\|u^{-1}|_{u(X)}\|$. One may take $D\le\sigma$ for every strictly admissible local constant $\sigma$.
\end{lemma}
\begin{proof}
Let the index set consist of all finite-dimensional local subspaces $F$, each equipped with a chosen isomorphism $U_F:F\to\ell_1^{\dim F}$ normalized so that
\[
 \|x\|\le\|U_Fx\|_1\le\sigma\|x\|\qquad(x\in F).
\]
Choose an ultrafilter containing all sets $\{F:F\supset F_0\}$, for finite-dimensional $F_0\subset X$. In the Banach-space ultraproduct of the spaces $\ell_1^{\dim F}$, define
\[
 ux=[(z_F(x))_F],\qquad z_F(x)=
 \begin{cases}U_Fx,&x\in F,\\0,&x\notin F.\end{cases}
\]
The map is linear because every finite set of vectors lies in the same $F$ on a set belonging to the ultrafilter; the two norm inequalities pass to the ultralimit. The Banach-lattice ultraproduct is an abstract $L_1$ lattice: its norm is additive on the positive cone. Kakutani's representation theorem \cite{Kak41} identifies it isometrically with an $L_1(\mu)$. Over $\mathbb C$, take the ultraproduct as a complex Banach lattice and apply the representation to its associated real lattice; its complexification gives complex $L_1(\mu)$. This proves the assertion without a separability assumption.
\end{proof}

We shall use Lindenstrauss' lifting principle in the following form: if $F$ is an $\mathcal L_1$-space and $Q:Y\to E$ is a surjection whose kernel is complemented in its own bidual, then every operator $F\to E$ lifts to $Y$ \cite[Proposition 2.1]{KP97}. In particular, for a reflexive $H\subset X$, every operator $X\to X/H$ lifts to $X$.

For a reflexive kernel, the compactness argument also controls the norm of the lifting. The quotient below carries its usual quotient norm.

\begin{lemma}\label{lem:contractive-lifting}
Let $F$ be an $\mathcal L_1$-space with a strictly admissible local constant $\sigma$, and let $H$ be a reflexive closed subspace of $Y$. Every $T:F\to Y/H$ has a lifting $L:F\to Y$ with $\|L\|\le\sigma\|T\|$. If $F=L_1(\mu)$, the lifting can be chosen with $\|L\|=\|T\|$.
\end{lemma}
\begin{proof}
This is the reflexive-kernel version of the proof in \cite[Proposition 2.1]{KP97}. A reflexive subspace is proximinal: a bounded minimizing sequence in an affine fiber of $q_H$ has a weakly convergent subsequence after translation into $H$, and the norm is weakly lower semicontinuous. Thus each $e\in Y/H$ has a representative of norm $\|e\|$.

For a finite-dimensional local subspace $E\subset F$, choose $U:E\to\ell_1^{\dim E}$ with $\|U^{-1}\|\le1$ and $\|U\|\le\sigma$. Lift each $TU^{-1}e_j$ to a representative $y_j$ of the same norm. Then
\[
 L_Ex=\sum_j(Ux)_j y_j,\qquad
 q_HL_E=T|_E,\qquad \|L_E\|\le\sigma\|T\|.
\]
For each $f\in F$ the set
\[
 K_f=\{y\in Y:q_Hy=Tf,\ \|y\|\le\sigma\|T\|\|f\|\}
\]
is nonempty and weakly compact: after translating a fiber representative, it is a bounded weakly closed subset of the reflexive space $H$. In the compact product $\prod_{f\in F}K_f$, impose all the closed conditions of additivity and scalar homogeneity. Every finite collection of conditions is satisfied by some $L_E$, because one local subspace $E$ contains all vectors involved. The finite intersection property gives a linear lifting with the required norm bound.

For $F=L_1(\mu)$, every $\sigma>1$ is an admissible local constant. Take such liftings with $\sigma\downarrow1$ and a cluster point in the product of the weakly compact fibers with bound $2\|T\|\|f\|$. The limit is linear and has norm at most $\|T\|$. The reverse inequality follows from $\|q_H\|\le1$.
\end{proof}

We also use two classical geometric facts. The spaces $\mathcal L_1$ have the Dunford--Pettis property \cite{Die80}: every weakly compact operator defined on such a space is completely continuous. For a complex $\mathcal L_1$-space, its underlying real space is a real $\mathcal L_1$-space: the realification of $\ell_1^n(\mathbb C)$ has distance at most $\sqrt2$ from $\ell_1^{2n}(\mathbb R)$. Thus the real Dunford--Pettis theorem also gives the complex assertion. The spaces $L_1(\mu)$ are weakly sequentially complete \cite{AK}; hence so are their closed subspaces and, by Lemma~\ref{lem:embedding}, all $\mathcal L_1$-spaces.

\begin{lemma}\label{lem:noreflexive}
No infinite-dimensional reflexive subspace of an $\mathcal L_1$-space is complemented. Every infinite-dimensional reflexive subspace $R$ of such a space contains a subspace $H$ isomorphic to $\ell_r$ for some $1<r\le2$. Thus $H$ is internally uniformly subprojective and $a_X(H)=\infty$.
\end{lemma}
\begin{proof}
If $P:X\to N$ is a projection onto a reflexive subspace, it is weakly compact, hence completely continuous. Its restriction to $N$ is the identity. Every bounded sequence in $N$ has a weakly convergent subsequence, which is then norm convergent. The unit ball of $N$ is therefore compact, and $N$ is finite-dimensional.

For real scalars, apply the theorem of Aldous to an infinite-dimensional separable subspace of $u(R)\subset L_1(\mu)$: it contains a copy of $\ell_r$, $1\le r\le2$ \cite{Ald81}; proofs are given in \cite{Mau80,Kri80}. Reflexivity excludes $r=1$. A separable subspace of any $L_1(\mu)$ is supported on a countable union of finite-measure sets; multiplying by a positive integrable density and restricting to a countably generated sigma algebra reduces the statement to a separable probability $L_1$ space, which embeds isometrically into $L_1(0,1)$.

For completeness, the real extraction theorem gives the same conclusion over $\mathbb C$. The underlying real space of complex $L_1(\mu)$ embeds into real $L_1(\mu\times\mathbb T)$ by
\[
 f\longmapsto \bigl((\omega,t)\mapsto\operatorname{Re}(e^{it}f(\omega))\bigr),
\]
whose norm is $2\|f\|_1/\pi$ for normalized Haar measure on $\mathbb T$. We obtain a real embedding $U:\ell_r(\mathbb R)\to R$. Write $\overline R$ for the conjugate Banach space. The map
\[
 W:\ell_r(\mathbb C)\longrightarrow R\oplus_\infty\overline R,
 \qquad W(a+ib)=(Ua+iUb,\ Ua-iUb)
\]
is complex linear and bounded below: both $Ua$ and $Ub$ are recovered from its coordinates, so
\[
 \|a+ib\|_r\le\|a\|_r+\|b\|_r
 \le2\|U^{-1}\|\,\|W(a+ib)\|.
\]
The two coordinate operators cannot both be strictly singular, since strictly singular operators form an operator ideal closed under addition \cite[Section IV.5]{Kat}. One coordinate is bounded below on an infinite-dimensional complex subspace of $\ell_r$. The gliding-hump construction in Proposition~\ref{prop:subproj} shows that every such subspace contains a copy of complex $\ell_r$. Consequently $R$ or $\overline R$ contains this copy; since $\ell_r$ is isometric to its conjugate, $R$ contains one in either case.

Transport this copy back to $X$. Proposition~\ref{prop:subproj} and isomorphism invariance give a finite internal subprojectivity constant for $H$. Every infinite-dimensional subspace of $H$ is reflexive, so the first paragraph gives $a_X(H)=\infty$.
\end{proof}

\begin{lemma}\label{lem:James}
If a Banach space contains a copy of $\ell_1$, then for every $t>1$ it contains a subspace $W$ with $d(W,\ell_1)<t$.
\end{lemma}
\begin{proof}
We recall the block argument underlying James' non-distortion theorem \cite{Jam64}. Let $U:\ell_1\to X$ be an embedding and put
\[
 b_n=\inf\{\|Uz\|:\|z\|_1=1,\ \operatorname{supp}z\subset[n,\infty)\},
 \qquad b=\lim_n b_n>0.
\]
Choose $n_0$ and $\delta>0$ with $(b+\delta)/b_{n_0}<t$. Successively choose finitely supported unit vectors $z_k$, with disjoint supports after $n_0$, such that $\|Uz_k\|<b+\delta$. This is possible on every tail by the definition of $b_n$ and density of finite supports. For finitely supported scalar sequences,
\[
 b_{n_0}\sum_k|a_k|\le\Big\|\sum_k a_kUz_k\Big\|
 \le(b+\delta)\sum_k|a_k|.
\]
The closed span of the $Uz_k$ has the required distortion. The argument works over both scalar fields and for the given norm of $X$.
\end{proof}

\begin{lemma}\label{lem:saturation}
Suppose that $u:X\to L_1(\mu)$ is an embedding of distortion $D$ and that $X$ contains no infinite-dimensional reflexive subspace. Then $\LUSP(X)\le D$.
\end{lemma}
\begin{proof}
Every $M\in\sub(X)$ contains $\ell_1$. Indeed, if it did not, Rosenthal's theorem \cite{Ros74} would give a weakly Cauchy subsequence of every bounded sequence in $M$. Weak sequential completeness would make that subsequence weakly convergent; by the Eberlein--\v Smulian theorem, $M$ would be reflexive, a contradiction. For complex scalars one can apply the real theorem: a real copy of $\ell_1$ yields a complex copy by the conjugate-space argument in the proof of Lemma~\ref{lem:noreflexive}, now with $r=1$.

Apply Lemma~\ref{lem:James} in $u(M)$ with its ambient $L_1$ norm to obtain a $W\subset u(M)$ with $d(W,\ell_1)<t$, where $1<t<4/3$. The proof of Dor's projection theorem \cite{Dor75} given by Maurey \cite[pp.~XXI.7--XXI.9]{Mau75} gives a projection $Q:L_1(\mu)\to W$ with
\begin{equation}\label{eq:Dor}
 \|Q\|\le B(t):=\frac{t^2}{4-3t}.
\end{equation}
Here is how this sufficient bound enters. The finite-dimensional estimate in that proof is $C(t)=t/(4-3t)$, uniformly in the dimension. Taking a weak-star ultralimit of the projections onto the spans of the first $n$ vectors gives an operator $L_1(\mu)\to W^{**}$ of norm at most $C(t)$ which is the identity on $W$. Transporting the canonical norm-one projection $\ell_1^{**}\to\ell_1$ by an isomorphism of distortion less than $t$ gives a projection $W^{**}\to W$ of norm less than $t$. Their composition gives \eqref{eq:Dor}. The bound applies to real and complex scalars; optimal constants are not needed.

Thus $u^{-1}Qu:X\to u^{-1}(W)\subset M$ is a projection of norm at most $DB(t)$. Since $B(t)\to1$ as $t\downarrow1$, we have $a_X(M)\le D$ for every $M$, and \eqref{eq:usp-def} proves the lemma.
\end{proof}

\begin{theorem}\label{thm:L1}
For an infinite-dimensional $\mathcal L_1$-space $X$, over either scalar field, the following are equivalent:
\begin{enumerate}[label=\textup{(\roman*)},nosep]
\item $\LSF(X)<\infty$;
\item $\LUSP(X)<\infty$;
\item $X$ is subprojective;
\item $X$ contains no infinite-dimensional reflexive subspace.
\end{enumerate}
Moreover, with the value $\infty$ allowed,
\begin{equation}\label{eq:L1-equality}
 \LSF(X)=\LUSP(X).
\end{equation}
If these constants are finite, their common value is at most the distortion of any embedding of $X$ into an $L_1(\mu)$, hence at most every strictly admissible local $\mathcal L_1$ constant. If they are infinite, there is an $H\subset X$ for which the identities in \eqref{eq:norm-eps} satisfy \eqref{eq:root} and \eqref{eq:sqrt}.
\end{theorem}
\begin{proof}
(ii)$\Rightarrow$(i) is Proposition~\ref{prop:subproj}, and (ii)$\Rightarrow$(iii) is immediate. For (iii)$\Rightarrow$(iv), subprojectivity applied to an infinite-dimensional reflexive subspace would produce such a subspace complemented in $X$, contradicting Lemma~\ref{lem:noreflexive}. For (iv)$\Rightarrow$(ii), combine Lemmas~\ref{lem:embedding} and \ref{lem:saturation}.

For (i)$\Rightarrow$(iv), suppose $R\in\sub(X)$ is reflexive. Lemma~\ref{lem:noreflexive} gives $H\subset R$, isomorphic to $\ell_r$ for some $1<r\le2$, with a finite admissible internal constant $K$ and $a_X(H)=\infty$. In the lifting principle of \cite[Proposition 2.1]{KP97}, take $F=Y=X$, $E=X/H$, and $Q=q_H$. Its kernel $H$ is reflexive, hence complemented in $H^{**}$ by the identity. Thus $(X,H)$ has the lifting property. Theorem~\ref{thm:lifting}(b) gives $\LSF(X)=\infty$, a contradiction, and also gives the stated quantitative obstruction when (iv) fails.

It remains to prove \eqref{eq:L1-equality} in the finite case. Every $M\in\sub(X)$ contains $\ell_1$ by the first paragraph of Lemma~\ref{lem:saturation}. Fix $\eta>0$. Apply Lemma~\ref{lem:James}, this time with the original norm of $X$, to find $H\in\sub(M)$ with $d(H,\ell_1)<1+\eta$. Then $\LUSP(H)<1+\eta$. Since $X$ is subprojective, some $N\in\sub(H)$ is complemented in $X$. Heredity of the internal constant gives $\LUSP(N)\le\LUSP(H)<1+\eta$. Corollary~\ref{cor:complemented} now yields
\[
 a_X(M)\le a_X(N)\le\LUSP(N)\LSF(X)
 \le(1+\eta)\LSF(X).
\]
Take the supremum over $M$ and let $\eta\downarrow0$. The reverse inequality is Proposition~\ref{prop:subproj}. The bound by the embedding distortion is Lemma~\ref{lem:saturation}.
\end{proof}

\begin{corollary}\label{cor:L1mu}
If $L_1(\mu)$ is infinite-dimensional and isometric to $\ell_1(\Gamma)$, then $\LSF(L_1(\mu))=1$. If $\mu$ has a non-atomic measurable set of finite positive measure, then $\LSF(L_1(\mu))=\infty$. In particular, for a semi-finite measure $\mu$,
\[
 \LSF(L_1(\mu))=
 \begin{cases}1,&\mu\text{ is purely atomic},\\
 \infty,&\mu\text{ is not purely atomic}.
 \end{cases}
\]
\end{corollary}
\begin{proof}
The first case follows from Proposition~\ref{prop:subproj} and $\LSF\ge1$. On a non-atomic finite-measure set, successive bisections produce an independent sequence of Rademacher functions. The Khintchine inequalities show that their closed span in $L_1$ is isomorphic to $\ell_2$, so Theorem~\ref{thm:L1} gives the second case. The same conclusion holds for complex scalars. For semi-finite measures, the atomic part gives an $\ell_1(\Gamma)$ space, while failure of pure atomicity provides a non-atomic finite-measure part.
\end{proof}

The obstruction on $L_1(0,1)$ can be exhibited without choosing an unspecified copy of $\ell_r$.

\begin{corollary}\label{cor:L1explicit}
Let $G\subset L_1(0,1)$ be the closed span of independent standard Gaussian variables, circular complex Gaussian variables over $\mathbb C$. For $0<\varepsilon\le1$, put
\[
 |f|_\varepsilon=\max\{\varepsilon\|f\|_1,\ \dist(f,G)\}.
\]
The identity $J_\varepsilon:L_1\to(L_1,|\cdot|_\varepsilon)$ satisfies
\[
 \inn(J_\varepsilon)=\varepsilon,\qquad
 d_+(J_\varepsilon)\ge
 \frac{2}{1+\sqrt{5+4\varepsilon^{-1}}}.
\]
Consequently $\liminf_{\varepsilon\downarrow0}d_+(J_\varepsilon)/\sqrt\varepsilon\ge1$ and $\LSF(L_1(0,1))=\infty$.
\end{corollary}
\begin{proof}
Gaussian rotation invariance gives
$\|\sum a_ng_n\|_1=\gamma(\sum|a_n|^2)^{1/2}$, where $\gamma=\|g_1\|_1$. Thus $G$ is a Hilbert space with its inherited norm and admits the internal constant $K=1$. Lemma~\ref{lem:noreflexive} gives $a_{L_1}(G)=\infty$. Since $G$ is reflexive, Lemma~\ref{lem:contractive-lifting} gives the lifting constant $C=1$. Now use Theorem~\ref{thm:lifting}(b).
\end{proof}

\begin{remark}
The square-root lower bound follows from the quadratic term in Lemma~\ref{lem:key}. No matching upper bound is established here; the exact asymptotic size of this radius remains a question.
\end{remark}

\subsection{A conditional converse}

\begin{theorem}\label{thm:conditional}
Suppose there is $K<\infty$ such that every $M\in\sub(X)$ contains an $H\in\sub(M)$ that is uniformly subprojective with constant $K$ and such that $(X,H)$ has the lifting property. Then
\[
 \LSF(X)\le\LUSP(X)\le K\,\LSF(X).
\]
\end{theorem}
\begin{proof}
If $\LSF(X)=\infty$, the right inequality is automatic. Otherwise, for such $H\subset M$, Theorem~\ref{thm:lifting}(a) gives $a_X(M)\le a_X(H)\le K\LSF(X)$. Take the supremum over $M$ and use \eqref{eq:usp-def} and Proposition~\ref{prop:subproj}.
\end{proof}

The lifting constants may depend on $H$; no uniform bound on them is required, since Theorem~\ref{thm:lifting}(a) is asymptotic. The same conclusion holds if the hypothesis is stated as $\LUSP(H)\le K$: use admissible constants greater than $K$ and let them decrease to $K$.

For example, the hypothesis holds when $X$ is subprojective and every $M\in\sub(X)$ contains a subspace $D$-isomorphic to a Hilbert space, for one fixed $D$. Choose a further subspace complemented in $X$; its inherited norm has internal constant at most $D$, and complementation provides the lifting property. Theorem~\ref{thm:conditional} gives $\LUSP(X)\le D\LSF(X)$.

\subsection{A subprojective space with \texorpdfstring{$\LSF=\infty$}{infinite constant}}

Subprojectivity alone does not give a finite constant.

\begin{corollary}\label{cor:Lpn}
Let $2<p_1<p_2<\cdots$ with $p_n\to\infty$, and let $X=(\bigoplus_nL_{p_n}(0,1))_{\ell_2}$ be the real $\ell_2$-sum. Then $X$ is subprojective and $\LSF(X)=\infty$.
\end{corollary}
\begin{proof}
Subprojectivity is \cite[Proposition 2.1 and Remark 2.11]{OiSp15}. For $p>2$ let $G_p\subset L_p(0,1)$ be the closed span of independent standard Gaussian variables $(g_n)$. Then $\|\sum a_ng_n\|_p=\gamma_p(\sum|a_n|^2)^{1/2}$, and $Pf=\sum_n\langle f,g_n\rangle g_n$ is bounded with $\|Pf\|_p\le\gamma_p\|f\|_2\le\gamma_p\|f\|_p$ and is a projection onto $G_p$, so $G_p$ is a complemented Hilbertian subspace. The Gaussian projection estimate recorded in \cite[Remark 2.11]{OiSp15}, which attributes the estimate to Gordon--Lewis--Retherford \cite{GLR73}, says that every projection of $L_p$ onto $G_p$ has norm at least $c_*\sqrt p$, with $c_*>0$ absolute. The same holds for each $H_0\in\sub(G_p)$: an $L_2$-orthonormal basis $(h_n)$ of $H_0$ consists of independent standard Gaussian variables, the pair $(L_p(\sigma(h_n)),H_0)$ is isometric to $(L_p,G_p)$, and the conditional expectation onto $L_p(\sigma(h_n))$ has norm one. Placing $G_{p_n}$ in the $n$th coordinate of $X$ we get $a_X(G_{p_n})\ge c_*\sqrt{p_n}$, and Corollary~\ref{cor:complemented} gives $\LSF(X)\ge c_*\sqrt{p_n}$ for all $n$.
\end{proof}

\section{Quotient obstructions on H.I. domains}\label{sec:HI}

\subsection{The operator ideal and the role of the target norm}
A Banach space $X$ is \emph{hereditarily indecomposable} (H.I.) if no closed infinite-dimensional subspace is a topological direct sum of two closed infinite-dimensional subspaces. The classical characterization \cite[Theorem 4.2]{G07} gives
\begin{equation}\label{eq:HI-distance}
 \cL(X,Y)\setminus\Phi_+(X,Y)=\StrictS(X,Y),
 \qquad d_+(T)=\dist(T,\StrictS(X,Y))
\end{equation}
for every target $Y$ when $X$ is H.I. Here $\StrictS$ denotes the ideal of strictly singular operators. We use its small-restriction property: every $M\in\sub(X)$ contains, for each strictly singular $Q$ and each $\delta>0$, an $N\in\sub(M)$ with $\|Q|_N\|<\delta$ \cite[Section IV.5]{Kat}.

An H.I. space is not subprojective. Indeed, choose a closed infinite-dimensional subspace $M$ of infinite codimension. An infinite-dimensional $N\subset M$ complemented in $X$ would split $X$ into two infinite-dimensional summands. Every infinite-dimensional subspace of $X$ is itself H.I., so none is internally uniformly subprojective. These domains are therefore outside the witness hypothesis of Theorem~\ref{thm:conditional}.

\begin{proposition}\label{prop:HI-restriction}
If $X$ is H.I., then $\inn(T|_M)=\inn(T)$ for every $T:X\to Y$ and every $M\in\sub(X)$.
\end{proposition}
\begin{proof}
Only $\inn(T|_M)\le\inn(T)$ needs proof. Fix $A\in\sub(X)$ and $\delta>0$. Since $q_A:X\to X/A$ has infinite-dimensional kernel, \eqref{eq:HI-distance} makes it strictly singular. Choose $N\in\sub(M)$ with $\|q_A|_N\|<\delta$. For each $x\in N$, choose $y\in A$ with $\|x-y\|\le\delta\|x\|$. Then
\[
 \|Tx\|\le[\|T|_A\|+\delta(\|T|_A\|+\|T\|)]\|x\|.
\]
Let $\delta\downarrow0$ and take the infimum over $A$.
\end{proof}

\begin{corollary}\label{cor:HI-endomorphisms}
On a complex H.I. space, every endomorphism has a unique representation $T=\lambda I+S$, $S\in\StrictS(X)$, and
\[
 d_+^{X,X}(T)=\inn(T)=|\lambda|.
\]
\end{corollary}
\begin{proof}
The representation is the complex H.I. operator theorem \cite{GoMa93,G07}. The small-restriction property gives $\inn(R+S)=\inn(R)$ for strictly singular $S$. Strictly singular perturbations preserve upper semi-Fredholmness \cite[Section IV.5]{Kat}, so translation by $S$ also preserves $d_+$. Finally $\inn(\lambda I)=d_+(\lambda I)=|\lambda|$, by the Neumann series and the perturbation $\lambda I$ itself.
\end{proof}

The target in Problem~\ref{prob:converse} is allowed to change independently of the domain. In particular, Theorem~\ref{thm:renorm} gives
\begin{equation}\label{eq:HI-renorm}
 \LSF(X)=\sup_{|\cdot|\sim\|\cdot\|}
 \frac{\dist(J_{|\cdot|},\StrictS(X,X_{|\cdot|}))}{\inn(J_{|\cdot|})}.
\end{equation}
The preceding corollary concerns a fixed pair and does not determine this supremum.

\subsection{An exact formula and a quotient characterization}
Let $Q:X\to Z$ be a strictly singular contraction and put
\begin{equation}\label{eq:HI-graph}
 |x|_{\varepsilon,Q}=\max\{\varepsilon\|x\|,\|Qx\|\},
 \qquad J_{\varepsilon,Q}:X\to X_{\varepsilon,Q},\quad 0<\varepsilon\le1.
\end{equation}
All norms on the right-hand side of the next formula are those of the initial spaces.

\begin{proposition}\label{prop:HI-graph}
For an H.I. domain,
\begin{equation}\label{eq:HI-exact}
 \inn(J_{\varepsilon,Q})=\varepsilon,
 \qquad d_+(J_{\varepsilon,Q})=
 \inf_{S\in\StrictS(X)}\max\{\varepsilon\|I-S\|,\|Q-QS\|\}.
\end{equation}
For a fixed $Q$, the ratios $d_+(J_{\varepsilon,Q})/\varepsilon$ are bounded if and only if there are $S_n\in\StrictS(X)$ such that
\[
 \sup_n\|I-S_n\|<\infty,\qquad \|Q-QS_n\|\longrightarrow0.
\]
\end{proposition}
\begin{proof}
The floor gives $\inn\ge\varepsilon$, and the small-restriction property gives an infinite-dimensional subspace where the new norm equals $\varepsilon\|\cdot\|$. Equivalent target norms preserve strict singularity. Thus \eqref{eq:HI-distance} gives the infimum, since
\[
 \|J_{\varepsilon,Q}-S\|_{X\to X_{\varepsilon,Q}}
 =\max\{\varepsilon\|I-S\|,\|Q(I-S)\|\}.
\]
If the ratios have bound $C$, use $\varepsilon=1/n$ to choose $S_n$ with $\|I-S_n\|<C+1/n$ and $\|Q-QS_n\|<(C+1/n)/n$. Conversely, if $\|I-S_n\|\le A$, choose $n$ with $\|Q-QS_n\|\le A\varepsilon$ and use the formula.
\end{proof}

For $M\in\sub(X)$, with the initial quotient norm on $X/M$, define
\begin{equation}\label{eq:HI-bdef}
 b_X(M)=\inf\left\{\limsup_n\|I-S_n\|:
 S_n\in\StrictS(X),\ \|q_M(I-S_n)\|\longrightarrow0\right\},
\end{equation}
where the infimum of the empty set is $\infty$.

\begin{theorem}\label{thm:HI-quotients}
For every H.I. space $X$, with $\infty$ allowed,
\begin{equation}\label{eq:HI-b}
 \LSF(X)=\sup_{M\in\sub(X)}b_X(M).
\end{equation}
In particular, finiteness is equivalent to the existence of one $C<\infty$ such that every $M\in\sub(X)$ has a sequence as in \eqref{eq:HI-bdef} with $\sup_n\|I-S_n\|\le C$.
\end{theorem}
\begin{proof}
If $\LSF(X)=C_0<\infty$, use Proposition~\ref{prop:HI-graph} with $Q=q_M$ and $\varepsilon=1/n$. The quotient is strictly singular by \eqref{eq:HI-distance}. The resulting sequence satisfies
\[
 \|I-S_n\|<C_0+1/n,\qquad \|q_M(I-S_n)\|<(C_0+1/n)/n.
\]
Hence $b_X(M)\le C_0$.

Conversely, suppose $B=\sup_M b_X(M)<\infty$. Fix $T:X\to Y$, $M\in\sub(X)$ and $\delta>0$. Choose $R_n=I-S_n$ with $\limsup\|R_n\|\le b_X(M)+\delta$ and $a_n=\|q_MR_n\|\to0$. For each $x$, choose $m\in M$ such that $\|R_nx-m\|\le(a_n+1/n)\|x\|$. No linear selection is needed. It follows that
\[
 \|TR_n\|\le\|T|_M\|\|R_n\|
 +(\|T|_M\|+\|T\|)(a_n+1/n).
\]
Since $TS_n$ is strictly singular, $d_+(T)\le\|TR_n\|$ by \eqref{eq:HI-distance}. Take the limit, let $\delta\downarrow0$, and then take the infimum over $M$. This gives $d_+(T)\le B\inn(T)$. The two implications also prove \eqref{eq:HI-b} for infinite values. A uniform limsup bound becomes a uniform bound on the sequences by discarding finitely many terms and allowing any $C>B$.
\end{proof}

\begin{corollary}\label{cor:HI-brai}
If $\LSF(X)<\infty$, the ideal $\StrictS(X)$ and its quotient by $\mathcal K(X)$ admit bounded right approximate identities.
\end{corollary}
\begin{proof}
For a finite family $A_1,\ldots,A_r\in\StrictS(X)$, choose $L\ge1$ with $\|A_j\|\le L$. The map $Qx=L^{-1}(A_1x,\ldots,A_rx)$ into the finite $\ell_\infty$-sum of copies of $X$ is a strictly singular contraction. For $\delta>0$ fixed, the finite comparison bound and \eqref{eq:HI-exact} give $S\in\StrictS(X)$ with
\[
 \|I-S\|<\LSF(X)+\delta,
 \qquad \max_j\|A_j-A_jS\|<L(\LSF(X)+\delta)\varepsilon.
\]
Direct the choices by finite families and required precision. They give a right approximate identity of norm at most $\LSF(X)+\delta+1$. Its image in the quotient has the same property.
\end{proof}

\subsection{Cosingularity gives an exact radius}
An endomorphism $S$ is strictly cosingular if $q_MS$ is not surjective for any closed infinite-codimensional $M\subset X$; write $\mathcal{SC}(X)$ for this class.

\begin{theorem}\label{thm:HI-cosingular}
Suppose $X$ is H.I. and $\StrictS(X)\subset\mathcal{SC}(X)$. For every closed $M$ of infinite dimension and codimension, the identity into
\begin{equation}\label{eq:HI-quotientnorm}
 |x|_{\varepsilon,M}=\max\{\varepsilon\|x\|,\dist(x,M)\}
\end{equation}
satisfies
\begin{equation}\label{eq:HI-radius-one}
 \inn(J_{\varepsilon,M})=\varepsilon,
 \qquad d_+(J_{\varepsilon,M})=1\quad(0<\varepsilon\le1).
\end{equation}
Consequently $\LSF(X)=\infty$.
\end{theorem}
\begin{proof}
The first equality also follows directly from $q_M|_M=0$ and the floor in the norm. For every $S\in\StrictS(X)$, $\|q_M-q_MS\|\ge1$. Indeed, if this norm were $r<1$, the isometry $q_M^*$ would give
\[
 \|(q_MS)^*f\|\ge(1-r)\|f\|\quad(f\in(X/M)^*).
\]
An operator whose adjoint is bounded below is surjective, contradicting the strict cosingularity of $S$. Formula~\eqref{eq:HI-exact} now gives $d_+\ge1$. The choice $S=0$ gives $d_+\le1$.
\end{proof}

We record a standard fact used to identify examples. A strictly singular endomorphism of an infinite-dimensional space is not lower semi-Fredholm. For every real $t\ne0$, $S-tI$ is Fredholm of index zero, by strictly singular perturbation stability. If $S$ were lower semi-Fredholm, local constancy of the semi-Fredholm index would give $\operatorname{ind}S=0$ as $t\to0$, forcing $S$ to be Fredholm, a contradiction \cite[Chapter IV]{Kat}.

\begin{corollary}\label{cor:HI-GM}
Theorem~\ref{thm:HI-cosingular} applies if $X$ is H.I. and quotient indecomposable, or if both $X$ and $X^*$ are H.I. In particular, it applies to the original Gowers--Maurey space.
\end{corollary}
\begin{proof}
A space is quotient indecomposable if none of its quotients is a topological direct sum of two closed infinite-dimensional subspaces. For such a range, each endomorphism is either lower semi-Fredholm or strictly cosingular \cite[Theorem 4.3]{G07}. The preceding fact excludes the first alternative for a strictly singular $S$.

If $X^*$ is H.I., the same fact and semi-Fredholm duality imply that $S^*$ is not upper semi-Fredholm. Thus $S^*$ is strictly singular by \eqref{eq:HI-distance} on $X^*$. A surjection $q_MS$ onto an infinite-dimensional quotient would make $S^*q_M^*$ bounded below, which is impossible. Hence $S$ is strictly cosingular.

Ferenczi proved that the original Gowers--Maurey space is quotient hereditarily indecomposable, and that its dual is H.I. \cite[Corollaries 21 and 22]{Fer99}.
\end{proof}

The domain keeps its initial norm throughout. Varying the equivalent target norm gives $d_+/\inn=1/\varepsilon$, whereas Corollary~\ref{cor:HI-endomorphisms} computes the ratio for the initial endomorphism norm. The fixed-pair product constructions of \cite{GM12} address the latter type of comparison.

\begin{corollary}\label{cor:HI-small-dual}
Let $X$ be complex H.I. If every endomorphism of $X^*$ is a scalar multiple of the identity plus a strictly singular operator, then \eqref{eq:HI-radius-one} holds. This includes the complex H.I. examples of Ferenczi whose duals are not H.I. \cite[Propositions 23--25]{Fer99}.
\end{corollary}
\begin{proof}
For $S\in\StrictS(X)$ write $S^*=\mu I+V$ with $V$ strictly singular. If $\mu\ne0$, $S^*$ and hence $S$ would be Fredholm, a contradiction. Thus $S^*$ is strictly singular, implying that $S$ is strictly cosingular as in the preceding proof. Apply Theorem~\ref{thm:HI-cosingular}.
\end{proof}

\begin{corollary}\label{cor:HI-nonSC}
If an H.I. space $X$ has $\LSF(X)<\infty$, then for each $M$ of infinite dimension and codimension there are bounded $S_n\in\StrictS(X)$ with
\[
 \|(I-S_n^*)|_{M^\perp}\|\longrightarrow0.
\]
For all sufficiently large $n$, $S_n^*$ is bounded below on $M^\perp$ and $q_MS_n$ is surjective. In particular $\StrictS(X)\not\subset\mathcal{SC}(X)$, and $X^*$ is not H.I.
\end{corollary}
\begin{proof}
Use Theorem~\ref{thm:HI-quotients} and the isometry of $q_M^*$ onto $M^\perp$. For large $n$, $\|(I-S_n^*)|_{M^\perp}\|<1$, so $S_n^*q_M^*$ is bounded below. Surjectivity follows by adjoint duality. The last assertion follows from Corollary~\ref{cor:HI-GM}.
\end{proof}

Over $\mathbb C$, finiteness also excludes the scalar-plus-strictly-singular property on $X^*$, by Corollary~\ref{cor:HI-small-dual}. Merely having a non-H.I. dual is therefore not enough to escape the obstruction.

\subsection{Compact products give a second obstruction}
\begin{proposition}\label{prop:HI-separation}
For the contraction $Q$ in \eqref{eq:HI-graph}, if
\[
 \delta_Q=\inf_{S\in\StrictS(X)}\|Q-QS\|>0,
\]
then $d_+(J_{\varepsilon,Q})\ge\delta_Q$ and $\LSF(X)=\infty$. It suffices that $Q$ be non-compact and $QS$ compact for every $S\in\StrictS(X)$.
\end{proposition}
\begin{proof}
The first conclusion is \eqref{eq:HI-exact}. In the sufficient case, $\delta_Q\ge\dist(Q,\mathcal K(X,Z))>0$, because the compact ideal is norm closed.
\end{proof}

\begin{theorem}\label{thm:HI-nilpotent}
Let $X$ be H.I. If every product of $m$ strictly singular endomorphisms is compact for some $m\ge1$, then $\LSF(X)=\infty$.
\end{theorem}
\begin{proof}
If $\StrictS(X)=\mathcal K(X)$, then $\StrictS(X)\subset\mathcal{SC}(X)$ and Theorem~\ref{thm:HI-cosingular} applies. Otherwise choose the largest length $k$ of a non-compact product. Such a $k$ exists and $1\le k\le m-1$. Choose $Q_0=S_1\cdots S_k$ non-compact. It is strictly singular, and $Q_0S$ is compact for every $S\in\StrictS(X)$, by maximality. Normalize $Q=Q_0/\|Q_0\|$ and apply Proposition~\ref{prop:HI-separation} with $Z=X$.
\end{proof}

The hypothesis says that $\StrictS(X)/\mathcal K(X)$ is nilpotent. The proof uses a non-zero element in its last non-zero layer, whose right products vanish. It does not require every strictly singular operator to be compact.

\begin{theorem}\label{thm:HI-reflexive-ideal}
If $X$ is H.I. and the Banach space $A=\StrictS(X)/\mathcal K(X)$ is reflexive, then $\LSF(X)=\infty$. In particular, this holds whenever the Calkin algebra $\cL(X)/\mathcal K(X)$ is finite-dimensional.
\end{theorem}
\begin{proof}
Suppose $\LSF(X)<\infty$. By Corollary~\ref{cor:HI-brai}, $A$ has a bounded right approximate identity $(e_\alpha)$. Reflexivity gives a weakly convergent subnet with limit $e\in A$. For each $a\in A$, bounded linearity of left multiplication and $ae_\alpha\to a$ in norm give $ae=a$. In particular $e^2=e$.

Choose $S\in\StrictS(X)$ representing $e$. The operator $I-S$ is Fredholm, so $1-e$ is invertible in the full Calkin algebra by Atkinson's theorem. But $e(1-e)=0$, hence $e=0$. Since $ae=a$ for all $a$, this forces $A=0$ and $\StrictS(X)=\mathcal K(X)$. Theorem~\ref{thm:HI-cosingular} then contradicts finiteness. The finite-dimensional assertion follows because $A$ is a closed subspace of the Calkin algebra.
\end{proof}

Here reflexivity refers to the quotient ideal, not to the domain. A hypothetical H.I. domain with finite $\LSF$ must therefore have a non-reflexive, non-nilpotent quotient ideal admitting a bounded right approximate identity.

\begin{corollary}\label{cor:HI-examples}
The constant $\LSF$ is infinite on the Argyros--Haydon space, on every closed infinite-dimensional subspace of the real Argyros--Motakis space $\mathfrak X_{\mathrm{ISP}}$, and on every closed infinite-dimensional subspace of the real H.I. spaces constructed in \cite{AM16}.
\end{corollary}
\begin{proof}
The scalar-plus-compact property of Argyros--Haydon \cite{AH11}, in either scalar field \cite{KL17}, gives $\StrictS(X)=\mathcal K(X)$: a non-zero scalar multiple of $I$ plus a compact operator is Fredholm and cannot be strictly singular. Apply Theorem~\ref{thm:HI-cosingular}.

For $\mathfrak X_{\mathrm{ISP}}$, Argyros--Motakis show that every infinite-dimensional closed subspace is H.I. and every product of three strictly singular endomorphisms on it is compact \cite[Introduction, theorem, (i); Theorem 5.19]{AM14}. Apply Theorem~\ref{thm:HI-nilpotent} with $m=3$. Their strictly singular ideals contain non-compact operators, so this case is distinct from scalar-plus-compact. The real H.I. constructions in \cite[Theorems 6.13 and 8.4]{AM16} have compact products of two strictly singular operators, also on their closed subspaces; use $m=2$.
\end{proof}

Proposition~\ref{prop:complemented-inheritance} also gives $\LSF(X)=\infty$ whenever $X$ contains a complemented subspace isomorphic to any of the spaces just excluded.

\subsection{A dual formula on complex H.I. spaces}
Let $X$ now be complex H.I. Put $B=\cL(X)$ and let $\varphi(T)=\lambda$ in the scalar-plus-strictly-singular representation. Corollary~\ref{cor:HI-endomorphisms} gives $\|\varphi\|=1$ and $\ker\varphi=\StrictS(X)$. For $M\in\sub(X)$, set
\[
 E_M=\cL(X,X/M),\quad F_M:B\to E_M,\quad F_M(A)=q_MA,
\]
\begin{equation}\label{eq:HI-eta}
 D_M=\overline{F_M^*(E_M^*)}^{\|\cdot\|},
 \qquad \eta_X(M)=\dist(\varphi,D_M).
\end{equation}

\begin{theorem}\label{thm:HI-dual}
With $1/0=\infty$,
\[
 b_X(M)=\frac1{\eta_X(M)},
 \qquad \LSF(X)=\frac1{\inf_{M\in\sub(X)}\eta_X(M)}.
\]
\end{theorem}
\begin{proof}
The operators $R=I-S$ with $S$ strictly singular are exactly those satisfying $\varphi(R)=1$. For a bounded admissible sequence $R_n$ and $f\in E_M^*$,
\[
 1\le\|\varphi-F_M^*f\|\|R_n\|+\|f\|\|F_MR_n\|.
\]
Taking the limit and then the infimum over $f$ gives $1\le\eta_X(M)\limsup\|R_n\|$. This proves the lower bound and the assertion when $\eta_X(M)=0$.

If $\eta=\eta_X(M)>0$, Hahn--Banach gives $u\in B^{**}$ with $\|u\|=1$, $u|_{D_M}=0$ and $u(\varphi)=\eta$. Set $v=u/\eta$. Then $F_M^{**}v=0$ and $v(\varphi)=1$. Goldstine gives a net $A_\alpha$ in the ball of $B$ of radius $1/\eta$, converging weak-star to $v$. The pairs $(F_MA_\alpha,\varphi(A_\alpha)-1)$ converge weakly to zero in $E_M\oplus\mathbb C$. By equality of weak and norm closures for convex sets, choose finite convex combinations $B_n$ with
\[
 \|B_n\|\le1/\eta,\qquad \|F_MB_n\|\to0,\qquad \varphi(B_n)\to1.
\]
Normalize $R_n=B_n/\varphi(B_n)$ for large $n$. These operators are admissible and $\limsup\|R_n\|\le1/\eta$. This proves the first equality; Theorem~\ref{thm:HI-quotients} gives the second.
\end{proof}

The closure in \eqref{eq:HI-eta} is a norm closure. For infinite-codimensional $M$, every endomorphism with range in $M$ is strictly singular: otherwise it would be upper semi-Fredholm by \eqref{eq:HI-distance}, hence Fredholm by the scalar-plus-strictly-singular representation, which its range excludes. Thus $\varphi$ annihilates $\ker F_M$ and belongs to the weak-star closure of $F_M^*(E_M^*)$. The distance to the norm closure $D_M$ is the quantity entering the comparison constant.

\begin{remark}\label{rem:HI-distortion}
Tomczak--Jaegermann proved that every H.I. space is arbitrarily distortable \cite[Corollary 3]{TJ96}. For an equivalent norm, its distortion is
\[
 \Delta_X(|\cdot|)=\inf_{M\in\sub(X)}
 \frac{\sup_{x\in M,\ \|x\|=1}|x|}{\inf_{x\in M,\ \|x\|=1}|x|}.
\]
Large distortion does not itself estimate the extension cost in Proposition~\ref{prop:extension}. Conversely, every graph norm \eqref{eq:HI-graph} has $\Delta_X(|\cdot|_{\varepsilon,Q})=1$: on an infinite-dimensional subspace where $\|Q\|<\varepsilon$, it equals $\varepsilon\|\cdot\|$. Hence the exact divergence $d_+/\inn=1/\varepsilon$ in Theorem~\ref{thm:HI-cosingular} occurs with distortion one. Persistent distortion is not necessary for this obstruction. Martin\'on \cite{Mar10} studies the relation between distortion and supermultiplicative operational quantities; the quotient argument above uses the surjectivity properties of $q_MS$.
\end{remark}

\section{Further questions}\label{sec:questions}
The results settle the comparison on $\mathcal L_1$ and exclude several H.I. classes, including the original Gowers--Maurey space. They leave Problem~\ref{prob:converse} open for general domains. Failure of the hypothesis of Theorem~\ref{thm:conditional} can reflect absence of witnesses, failure of lifting, or lack of a common internal bound $K$; it need not give a single subspace containing no witness at all.

\begin{problem}\label{prob:HI}
Is $\sup_{M\in\sub(X)}b_X(M)=\infty$ for every H.I. space $X$? Equivalently, is $\LSF(X)=\infty$ for every such domain? A finite value would require the quotient-surjectivity phenomenon in Corollary~\ref{cor:HI-nonSC} and a bounded right approximate identity in the non-reflexive, non-nilpotent quotient ideal $\StrictS(X)/\mathcal K(X)$. Over $\mathbb C$, it would require $\inf_M\eta_X(M)>0$ in Theorem~\ref{thm:HI-dual}. More generally, is $\LSF(X)=\infty$ whenever $X$ has an infinite-dimensional subspace containing no uniformly subprojective subspace?
\end{problem}

\begin{problem}\label{prob:Lp}
Is $\LSF(L_p(0,1))$ finite for $1<p<2$? These spaces are not subprojective \cite{OiSp15}, but the lifting principle used in Section~\ref{sec:converse} is specific to $\mathcal L_1$-spaces. For $2\le p<\infty$, $L_p(0,1)$ is uniformly subprojective \cite[Remark 2.11]{OiSp15}; determine $\LSF(L_p(0,1))$, which equals $1$ for $p=2$.
\end{problem}

\begin{problem}\label{prob:Linfty}
Determine the $\mathcal L_\infty$-spaces with $\LSF(X)<\infty$. For instance, is $\LSF(C[0,1])=\infty$?
\end{problem}

\subsection{The lower semi-Fredholm side}

For the lower semi-Fredholm side, with $\inn_-(T)=\inf_U\|Q_UT\|$ over closed infinite-codimensional $U\subset Y$ and $d_-(T)=\dist(T,\cL(X,Y)\setminus\Phi_-(X,Y))$, the inequality $\inn_-\le d_-$, the failure of equivalence, and the Hilbert-target case are in \cite[Section 3]{GM12}. The natural counterpart of Theorem~\ref{thm:L1} would concern a fixed target and uniform superprojectivity, with $\mathcal L_\infty$-spaces in place of $\mathcal L_1$-spaces and extension in place of lifting.

\begin{problem}
For which infinite-dimensional Banach spaces $Y$ is there a finite constant $C_Y$ such that $d_-^{X,Y}(T)\le C_Y\inn_-(T)$ for all domains $X$ and all $T:X\to Y$? Does a corresponding geometric characterization hold for $\mathcal L_\infty$ targets?
\end{problem}

\section{Conclusion}

The extension formula identifies the geometric issue behind the comparison of $\inn$ with the upper semi-Fredholm radius. A small restriction bounds the radius when it admits an extension of controlled norm. Uniform subprojectivity provides these extensions through projections. The lifting argument recovers projections in the reverse direction, from identities into quotient renormings, even when the witness is uncomplemented. On $\mathcal L_1$-spaces it yields both a characterization and equality of the optimal constants: reflexive subspaces force an infinite comparison constant, while their absence gives almost isometric $\ell_1$ witnesses and uniform projections.

On H.I. domains, quotient maps detect an obstruction that does not depend on uniformly subprojective witnesses. Cosingularity forces their distance from compositions $q_MS$ to be at least one, giving the exact radius $d_+=1$ while $\inn=\varepsilon$. This excludes the original Gowers--Maurey space and the other duality examples considered here as domains with finite comparison constant. Compact products and reflexive quotient ideals extend the conclusion to further H.I. constructions. The inheritance result carries each obstruction to spaces containing these examples as complemented subspaces.

For arbitrary H.I. spaces, the quotient characterization leaves a concrete problem: can the maps $q_M$ be approximated by $q_MS_n$ with one common bound on $\|I-S_n\|$? In the complex case the dual distance formula expresses exactly the bound required. Resolving this approximation problem would settle the comparison on all H.I. domains. The general converse, the case $L_p(0,1)$ for $1<p<2$, and the fixed-target lower semi-Fredholm problem remain open. The Gaussian renormings on $L_1(0,1)$ also leave a quantitative question: their square-root lower bound is explicit, but a matching asymptotic upper bound has yet to be found.

\section*{Acknowledgements}
The author thanks Manuel Gonz\'alez for comments on an earlier formulation of the comparison problem and for drawing attention to the distortion of H.I. spaces and the work of Tomczak--Jaegermann.

\section*{Declarations}
\textbf{Funding.} The author received no specific funding for this work.\\
\textbf{Competing interests.} The author declares no competing interests.\\
\textbf{Data availability.} No data sets were generated or analysed.\\
\textbf{Use of generative AI.} Generative AI tools assisted with drafting and preliminary mathematical checks. The author is responsible for the content and its final verification.


\begin{thebibliography}{99}\small
\bibitem{CR23} D.~Y.~Chen, Y.~B.~Ruan, Quantitative characterizations of bounded compact approximation property by semi-Fredholm operators, \emph{Acta Math. Sin. (Engl. Ser.)} 39 (2023), no.~6, 1115--1128.
\bibitem{AH11} S.~A.~Argyros, R.~G.~Haydon, A hereditarily indecomposable $\mathcal L_\infty$-space that solves the scalar-plus-compact problem, \emph{Acta Math.} 206 (2011), 1--54.
\bibitem{AM14} S.~A.~Argyros, P.~Motakis, A reflexive hereditarily indecomposable space with the hereditary invariant subspace property, \emph{Proc. London Math. Soc.} 108 (2014), 1381--1416.
\bibitem{AM16} S.~A.~Argyros, P.~Motakis, A dual method of constructing hereditarily indecomposable Banach spaces, \emph{Positivity} 20 (2016), 625--662.
\bibitem{Fer99} V.~Ferenczi, Quotient hereditarily indecomposable Banach spaces, \emph{Canad. J. Math.} 51 (1999), no.~3, 566--584.
\bibitem{G07} M.~Gonz\'alez, Banach spaces with small Calkin algebras, \emph{Banach Center Publ.} 75 (2007), 159--170.
\bibitem{GoMa93} W.~T.~Gowers, B.~Maurey, The unconditional basic sequence problem, \emph{J.\ Amer.\ Math.\ Soc.} 6 (1993), no.~4, 851--874.
\bibitem{KL17} T.~Kania, N.~J.~Laustsen, Ideal structure of the algebra of bounded operators acting on a Banach space, \emph{Indiana Univ. Math. J.} 66 (2017), 1019--1043.
\bibitem{Mar10} A.~Martin\'on, Distortion of Banach spaces and supermultiplicative operational quantities, \emph{J. Math. Anal. Appl.} 363 (2010), 655--662.
\bibitem{TJ96} N.~Tomczak-Jaegermann, Banach spaces of type $p$ have arbitrarily distortable subspaces, \emph{Geom.\ Funct.\ Anal.} 6 (1996), no.~6, 1074--1082.
\bibitem{AK} F.~Albiac, N.~J.~Kalton, \emph{Topics in Banach Space Theory}, 2nd ed., Graduate Texts in Mathematics 233, Springer, Cham, 2016.
\bibitem{Ald81} D.~J.~Aldous, Subspaces of $L^1$, via random measures, \emph{Trans.\ Amer.\ Math.\ Soc.} 267 (1981), no.~2, 445--463.
\bibitem{AAG92} J.~A.~\'Alvarez, T.~\'Alvarez, M.~Gonz\'alez, The gap between subspaces and perturbation of non-semi-Fredholm operators, \emph{Bull.\ Austral.\ Math.\ Soc.} 45 (1992), no.~3, 369--376.
\bibitem{Dor75} L.~E.~Dor, On projections in $L_1$, \emph{Ann.\ of Math.\ (2)} 102 (1975), no.~3, 463--474.
\bibitem{Die80} J.~Diestel, A survey of results related to the Dunford--Pettis property, \emph{Contemp.\ Math.} 2 (1980), 15--60.
\bibitem{GM95} M.~Gonz\'alez, A.~Martin\'on, Operational quantities characterizing semi-Fredholm operators, \emph{Studia Math.} 114 (1995), no.~1, 13--27. DOI: 10.4064/sm-114-1-13-27.
\bibitem{GM12} M.~Gonz\'alez, A.~Martin\'on, Quantities characterizing semi-Fredholm operators and perturbation radii, \emph{J.\ Math.\ Anal.\ Appl.} 390 (2012), no.~1, 362--367. DOI: 10.1016/j.jmaa.2012.01.048.
\bibitem{GO86} M.~Gonz\'alez, V.~M.~Onieva, On the instability of non-semi-Fredholm operators under compact perturbations, \emph{J.\ Math.\ Anal.\ Appl.} 114 (1986), no.~2, 450--457.
\bibitem{GLR73} Y.~Gordon, D.~R.~Lewis, J.~R.~Retherford, Banach ideals of operators with applications, \emph{J.\ Funct.\ Anal.} 14 (1973), 85--129. DOI: 10.1016/0022-1236(73)90031-1.
\bibitem{Jam64} R.~C.~James, Uniformly non-square Banach spaces, \emph{Ann.\ of Math.\ (2)} 80 (1964), 542--550.
\bibitem{Kak41} S.~Kakutani, Concrete representation of abstract $(L)$-spaces and the mean ergodic theorem, \emph{Ann. of Math. (2)} 42 (1941), no.~2, 523--537.
\bibitem{KP97} N.~J.~Kalton, A.~Pe\l czy\'nski, Kernels of surjections from $\mathcal L_1$-spaces with an application to Sidon sets, \emph{Math.\ Ann.} 309 (1997), no.~1, 135--158.
\bibitem{Kat} T.~Kato, \emph{Perturbation Theory for Linear Operators}, Grundlehren der mathematischen Wissenschaften 132, Springer, Berlin, 1966.
\bibitem{Kri80} J.~L.~Krivine, Plongement de $\ell^p$ dans certains espaces de Banach, \emph{S\'eminaire d'analyse fonctionnelle, \'Ecole Polytechnique} (1979--1980), expos\'e 12, 1--5. \url{https://www.numdam.org/item/SAF_1979-1980____A10_0/}.
\bibitem{Lin64} J.~Lindenstrauss, On a certain subspace of $\ell_1$, \emph{Bull.\ Acad.\ Polon.\ Sci.\ S\'er.\ Sci.\ Math.\ Astronom.\ Phys.} 12 (1964), 539--542.
\bibitem{Mau80} B.~Maurey, Tout sous-espace de $L^1$ contient un $\ell_p$, d'apr\`es D.~Aldous, \emph{S\'eminaire d'analyse fonctionnelle, \'Ecole Polytechnique} (1979--1980), expos\'es 1--2, 1--13. \url{https://www.numdam.org/item/SAF_1979-1980____A1_0/}.
\bibitem{Mau75} B.~Maurey, Projections dans $L^1$, d'apr\`es L.~Dor, \emph{S\'eminaire Maurey--Schwartz, \'Ecole Polytechnique} (1974--1975), expos\'e 21, 1--9. \url{https://www.numdam.org/item/SAF_1974-1975____A20_0/}.
\bibitem{OiSp15} T.~Oikhberg, E.~Spinu, Subprojective Banach spaces, \emph{J.\ Math.\ Anal.\ Appl.} 424 (2015), no.~1, 613--635. DOI: 10.1016/j.jmaa.2014.11.008.
\bibitem{Ros74} H.~P.~Rosenthal, A characterization of Banach spaces containing $\ell^1$, \emph{Proc.\ Nat.\ Acad.\ Sci.\ U.S.A.} 71 (1974), 2411--2413.
\bibitem{Sch72} M.~Schechter, Quantities related to strictly singular operators, \emph{Indiana Univ.\ Math.\ J.} 21 (1972), no.~11, 1061--1071.
\bibitem{Whi64} R.~J.~Whitley, Strictly singular operators and their conjugates, \emph{Trans.\ Amer.\ Math.\ Soc.} 113 (1964), 252--261.
\bibitem{Zem84} J.~Zem\'anek, Geometric characteristic of semi-Fredholm operators and their asymptotic behaviour, \emph{Studia Math.} 80 (1984), no.~3, 219--234.
\bibitem{Zem85} J.~Zem\'anek, The stability radius of a semi-Fredholm operator, \emph{Integral Equations Operator Theory} 8 (1985), 137--144.
\end{thebibliography}
\end{document}